\documentclass[a4paper,12pt]{amsart}
\pdfoutput=1
\usepackage{microtype}
\usepackage{amssymb,latexsym}
\usepackage[margin=3.25cm]{geometry}
\usepackage{mathtools}
\usepackage{enumitem}
\usepackage{mathrsfs}
\usepackage[dvipsnames]{xcolor}
\usepackage{hyperref}
\hypersetup{
    colorlinks=true,
    linkcolor=black, 
    citecolor=black} 

\DeclareMathOperator{\Ima}{Im}

\DeclareMathOperator{\Rank}{rank}
\DeclareMathOperator{\Span}{span}

\DeclareMathOperator{\Sign}{sgn}
\DeclareMathOperator{\Alt}{Alt}

\theoremstyle{plain}
\newtheorem{theorem}{Theorem}[section]
\newtheorem{proposition}[theorem]{Proposition}
\newtheorem{lemma}[theorem]{Lemma}
\newtheorem{corollary}[theorem]{Corollary}
\theoremstyle{remark}
\newtheorem{remark}[theorem]{Remark}
\theoremstyle{definition}

\newtheorem{problem}[theorem]{Problem}
\newtheorem{definition}[theorem]{Definition}

\begin{document}
\title[The Cauchy problem for rank-one submanifolds]{The geometric Cauchy problem \\ for rank-one submanifolds}

\author{Matteo Raffaelli}
\address{Department of Applied Mathematics and Computer Science\\
	Technical University of Denmark\\
	Matematiktorvet, Building 303B\\
	2800 Kongens Lyngby\\
	Denmark}
\curraddr{School of Mathematics, Georgia Institute of Technology, Atlanta, Georgia 30332}
\email{raffaelli@math.gatech.edu}


\date{September 8, 2024}
\begin{abstract}
Given a smooth distribution $\mathscr{D}$ of $m$-dimensional planes along a smooth regular curve $\gamma$ in $\mathbb{R}^{m+n}$, we consider the following problem: to find an $m$-dimensional rank-one submanifold of $\mathbb{R}^{m+n}$, that is, an $(m-1)$-ruled submanifold with constant tangent space along the rulings, such that its tangent bundle along $\gamma$ coincides with $\mathscr{D}$. In particular, we give sufficient conditions for the local well-posedness of the problem, together with a parametric description of the solution.
\end{abstract}
\maketitle

\tableofcontents

\section{Introduction and main result} \label{Introduction&MainResult}

Given a smooth $(m+n)$-manifold $Q^{m+n}$ and some class $\mathcal{A}^{m}$ of $m$-dimensional embedded submanifolds of $Q^{m+n}$, we can formulate the \textit{geometric Cauchy problem} for the class $\mathcal{A}^{m}$ as follows. 

\begin{problem} \label{IMR-PR1}
Let $\gamma \colon I \to Q^{m+n}$ be a smooth regular embedded curve in $Q^{m+n}$, and let $\mathscr{D}$ denote a smooth distribution of rank $m$ along $\gamma$ such that $\dot{\gamma}(t) \in \mathscr{D}_{t}$ for all $t \in I$. Find all members of $\mathcal{A}^{m}$ containing $\gamma$ and whose tangent bundle along $\gamma$ is precisely $\mathscr{D}$.
\end{problem}

\begin{remark}
In case $Q^{m+n}$ is a Riemannian manifold, let $\mathscr{D}^{\perp}_{t}$ be the orthogonal complement of $\mathscr{D}_{t}$ in the tangent space $T_{\gamma(t)}Q^{m+n}$. Then Problem \ref{IMR-PR1} is equivalent to finding all members of $\mathcal{A}^{m}$ containing $\gamma$ and whose \emph{normal} bundle along $\gamma$ coincides with the \textit{orthogonal distribution} $\mathscr{D}^{\perp} = \bigcup_{t} \mathscr{D}^{\perp}_{t}$.
\end{remark}

This problem has its roots in the classical Bj\"{o}rling problem for minimal surfaces in $\mathbb{R}^{3}$ and has recently been examined for several combinations of $Q^{2+n}$ and $\mathcal{A}^{2}$, generally with $n=1$; see e.g.\ \cite{brander2018, cintra2016, martinez2015, brander2013, aledo2009}.

In particular, in a joint work with Irina Markina \cite{markina2019}, the author has studied the case of developable, i.e., flat, hypersurfaces in $\mathbb{R}^{m+1}$ (the case $m=2$ is well-known \cite[p.~195--197]{docarmo1976}). We showed that, so long as the normal curvature of $\gamma$ is never vanishing, a solution exists, is locally unique, and may be constructed using a method alternative to the classical Gauss parametrization \cite{dajczer1985}; see Appendix \ref{AltConstr}.

The purpose of this article is twofold. On one hand, we aim to give a new and simpler proof of the main theorem in \cite{markina2019}. At the same time, we intend to generalize such result to the whole class of rank-one submanifolds of $\mathbb{R}^{m+n}$. These are precisely the ruled submanifolds without planar points and whose induced metric is flat (Theorem \ref{DS-TH1.4}).

In order to state our main theorem, set $Q^{m+n}=\mathbb{R}^{m+n}$. Let $\pi^{\top}$ and $\pi^{\perp}$ denote the orthogonal projections onto $\mathscr{D}$ and $\mathscr{D}^{\perp}$, respectively. Let $\overline{D}_{t}$ be the Euclidean covariant derivative along $\gamma$.

\begin{theorem} \label{IMR-TH3}
Assume that the function $\pi^{\perp}(\overline{D}_{t}\dot{\gamma})$ is never zero. The geometric Cauchy problem for rank-one submanifolds of $\mathbb{R}^{m+n}$ has a solution if and only if the linear map $\rho_{t} \colon \mathscr{D}^{\perp}_{t} \to \mathscr{D}_{t}$ defined by $\nu \mapsto \pi^{\top}(\overline{D}_{t}\nu)$ has rank one for every $t \in I$. Moreover:
\begin{enumerate}[font=\upshape,label=(\roman*)]
\item \label{item1} The solution is unique in the following sense: if $M_{1}$ and $M_{2}$ are two solutions of the Cauchy problem, then they coincide on an open set containing $\gamma(I)$.
\item \label{item2} In such neighborhood, the unique solution $M$ satisfies $M = \gamma + \mathscr{D} \cap (\Ima \rho)^{\perp}$.
\item \label{item3} When $I$ is closed, $M$ can be parametrized as follows: Let $(E_{1}, \dotsc, E_{m})$ be a smooth orthonormal frame for $\mathscr{D}$ satisfying $E_{1}=\dot{\gamma}$. Choose a smooth section $N$ of $\mathscr{D}^{\perp}$ such that $\pi^{\top}(\overline{D}_{t} N)$ is nonvanishing. For any $j = 1, \dotsc, m-1$, let 
\begin{equation*}
X_{j}=\left(\overline{D}_{t}E_{j+1} \cdot N \right) E_{1} - \left(\overline{D}_{t}E_{1} \cdot N \right) E_{j+1}.
\end{equation*}
For sufficiently small $\varepsilon >0$, let $\sigma \colon I \times (-\varepsilon, \varepsilon)^{m-1} \to \mathbb{R}^{m+n}$ be defined by
\begin{equation*}
\sigma(t,u^{1},\dotsc,u^{m-1})=\gamma(t) + u^{1}X_{1}(t) + \dotsb + u^{m-1}X_{m-1}(t).
\end{equation*}
Then, $M = \Ima\sigma$.
\end{enumerate}
\end{theorem}  

\begin{remark}
In the definition of $\rho_{t}$, the vector $\nu$ is understood to be extended arbitrarily to a smooth local section $N$ of $\mathscr{D}^{\perp}$. Well-definedness of $\rho_{t}$ is easily seen, for, denoting $E_{i}(t)$ by $e_{i}$,
\begin{align*}
\pi^{\top} (\overline{D}_{t}N)(t) &=\left(\overline{D}_{t}N(t) \cdot e_{1}\right)e_{1} + \dotsb +\left(\overline{D}_{t}N(t) \cdot e_{m} \right)e_{m}\\
&= -\left(\nu \cdot \overline{D}_{t}E_{1}(t)\right)e_{1} - \dotsb -\left(\nu \cdot \overline{D}_{t}E_{m}(t) \right)e_{m}.
\end{align*}
\end{remark}

\begin{remark}
If $I$ is not closed, then $M$ can be parametrized in essentially the same way as presented in \ref{item3}, by simply allowing the domain of $\sigma$ to be a subset of $\mathbb{R}^{m}$ of the form $\bigcup_{t\in I} \{t\}\times (-\varepsilon(t), \varepsilon(t))^{m-1}$. Clearly, if $\pi^{\perp}(\overline{D}_{t}\dot{\gamma})$ is nonzero in the closure of $I$, then one can still choose $\varepsilon$ to be constant.
\end{remark}

\begin{remark} \label{IMR-RM4}
The existence condition may be equivalently stated as follows: the function $\pi^{\perp}(\overline{D}_{t}\dot{\gamma})$ is never zero and there exists a smooth orthonormal frame $(N_{1}^{\ast},\dotsc,N_{n}^{\ast})$ for $\mathscr{D}^{\perp}$ such that $\pi^{\top}(\overline{D}_{t}N^{\ast}_{k})=0$ for all but one value of $k\in \{1,\dotsc,n\}$.
\end{remark}

Note that, when $n=1$, the nonvanishing of $\pi^{\perp}(\overline{D}_{t}\dot{\gamma})(t)$ becomes a sufficient condition for the rank of $\rho_{t}$ to be one, and we thus retrieve Theorem 1.1 in \cite{markina2019}.

The paper is organized as follows. In section \ref{Preliminaries} we review some background material. In section \ref{DevelopabilityCondition} we derive a simple criterion for discerning when a parametrized ruled submanifold is developable. Such criterion, extending a well-known result of Yano \cite{yano1944}, is of independent interest. In section \ref{ProofMainTh}, using an approach based on Grassmannians, we prove Theorem \ref{IMR-TH3}. In section \ref{CodimensionReduction} we present a sufficient condition for the solution to be a hypersurface in substantial codimension. In section \ref{Application} we apply our main result to the problem of approximating---locally along a curve---a given submanifold by a single-rank one. In section \ref{Description} we show that every rank-one submanifold can be described, in a neighborhood of a point, by $m+n-1$ smooth functions. Finally, in section \ref{DevDistr} we prove that any full curve can be extended to a full rank-one submanifold. There follow two appendixes: The first indicates a different method for proving Theorem \ref{IMR-TH3}. The second reviews a simpler parametrization, available in the case where $n=1$.

\subsection*{Notation}

Throughout the paper the integers $i,j,k$ satisfy $i \in \{1, \dotsc, m\}$, $j \in \{1, \dotsc, m-1\}$, and $k \in \{1, \dotsc, n\}$, where $m \geq 2$ and $n \geq 1$. Note that we always use Einstein summation convention.

\section{Preliminaries} \label{Preliminaries}

This section contains some preliminaries needed for the proof of Theorem \ref{IMR-TH3}.

\subsection{The wedge product}
Let $V$ be an $d$-dimensional real vector space, and let $V^{\ast}$ be its dual space. A \textit{tensor of type} $(l,r)$ on $V$ is a multilinear map
\begin{equation*}
F \colon \left(V^{\ast}\right)^{l} \times V^{r} \to \mathbb{R}.
\end{equation*}
The set of all such tensors---which is of course a vector space under pointwise addition and scalar multiplication---we denote by $T^{(l,r)}(V)$.

Recall that a multilinear map is called \textit{alternating} if its value changes sign whenever two arguments are interchanged. In particular, an alternating tensor of type $(0,r)$ is called an $r$-\textit{covector} on $V$, whereas one of type $(l,0)$ is an $l$-\textit{vector} on $V$. As usual, the sets of all $l$-vectors is denoted by $\Lambda^{l}(V)$, and we let $\Lambda(V) = \Lambda^{1}(V) \oplus \dotsb \oplus \Lambda^{d}(V)$. 

Given $\lambda \in \Lambda^{r}(V)$ and $\theta \in \Lambda^{l}(V)$, we define the \textit{wedge product} $\lambda \wedge \theta$ to be the following $(r+l)$-vector:
\begin{equation*}
\lambda \wedge \theta = \frac{(r+l)!}{r!\,l!} \Alt(\lambda \otimes \theta),
\end{equation*}
where $\Alt$ denotes alternation \cite[p.~351]{lee2013} and $\otimes$ is the ordinary tensor product. Being bilinear, the wedge product turns the vector space $\Lambda(V)$ into an (associative, anticommutative graded) algebra, called the \textit{exterior algebra} of $V$.

Given $v_{1},\dotsc,v_{l}\in V$ and $\eta^{1},\dotsc, \eta^{l}\in V^{\ast}$, one can check that
\begin{equation*}
v_{1} \wedge \dotsb \wedge v_{l} (\eta^{1}, \dotsc, \eta^{l}) = \det (\eta^{\alpha}(v_{\beta})) .
\end{equation*}
Moreover, we have the following lemma.
\begin{lemma}[{\cite[Problem~14-4]{lee2013}}]
An $l$-tuple $(v_{1}, \dotsc, v_{l})$ of elements of $V$ is linearly dependent if and only if $v_{1} \wedge \dotsb \wedge v_{l}=0$. Moreover, two $l$-tuples $(v_{1}, \dotsc, v_{l})$ and $(w_{1},\dotsc,w_{l})$ have the same span if and only if there exists a nonzero real number $\lambda$ such that
\begin{equation*}
v_{1} \wedge \dotsb \wedge v_{l} = \lambda \, w_{1} \wedge \dotsb \wedge w_{l}.
\end{equation*}
\end{lemma}

\subsection{\for{toc}{Grassmannians}\except{toc}{Grassmannians \cite[p.~82--83]{tu2011}}}
The \textit{Grassmannian} $G(l,V)$ is the set of all $l$-dimensional linear subspaces of $V$. Once a basis of $V$ has been chosen, we may identify $G(l,V)$ with the quotient $\mathcal{A}^{l\times d}/{\sim}$, where $\mathcal{A}^{l\times d}$ denotes the set of real $l \times d$ matrices of rank $l$,
\[
\mathcal{A}^{l\times d} = \{  A \in \mathbb{R}^{l \times d} \mid \Rank A = l  \} ,
\]
and $\sim$ the equivalence relation
\[ 
A \sim B \: \Leftrightarrow \: \text{there is a matrix $g \in \mathrm{GL}(l,\mathbb{R})$ such that $B = gA$.}
\]
Note that $A \sim B$ if and only if $A$ and $B$ have the same row space. 

One may show that $\mathcal{A}^{l\times d}/{\sim}$, with the quotient topology, is a compact topological manifold of dimension $l(d-l)$. In fact, it has a natural smooth structure, as we now explain.

Let $\pi$ be the canonical projection $\mathcal{A}^{l\times d} \to \mathcal{A}^{l\times d}/{\sim}$, and let $J$ be any strictly ascending multi-index $1 \leq i_{1} < \dotsb < i_{l} \leq d$ of length $l$. For $A \in \mathcal{A}^{l\times d}$, denote by $A_{J}$ the $l\times l$ submatrix of $A$ consisting of the $i_{1}\text{th}, \dotsc, i_{l}\text{th}$ columns of $A$. Then define
\[
V_{J} = \{A \in \mathcal{A}^{l\times d} \mid \det A_{J} \neq 0\}
\]
and $\tilde{\phi}_{J} \colon V_{J}\to \mathbb{R}^{l \times(d-l)}$,
\[
\tilde{\phi}_{J}(A)=(A^{-1}_{J}A)_{J'},
\]
where $(\;)_{J'}$ denotes the $l \times(d-l)$ submatrix obtained from the complement $J'$ of the multi-index $J$. Finally, let $U_{J} = V_{J}/{\sim}$, and $\phi_{J}$ such that $\tilde{\phi}_{J} = \phi_{J}\circ \pi$. It is standard to prove that $\{(U_{J}, \phi_{J})\}$ is a smooth atlas for $\mathcal{A}^{l\times d}/{\sim}$.

\subsection{Distributions along curves}

Let $\gamma$ be a smooth regular embedded curve $I \to \mathbb{R}^{m+n}$. Without loss of generality, we may assume $\gamma$ to be unit-speed. Recall that the \textit{ambient tangent bundle} $T\mathbb{R}^{m+n}\lvert_{\gamma}$ {over} $\gamma$ is the smooth vector bundle over $I$ defined as the disjoint union of the tangent spaces of $\mathbb{R}^{m+n}$ at all points of $\gamma(I)$: 
\begin{equation*}
T\mathbb{R}^{m+n}\lvert_{\gamma} \:= \bigsqcup_{t \,\in \, I} T_{\gamma(t)}\mathbb{R}^{m+n}.
\end{equation*}
We define a \textit{distribution of rank $m$ along} $\gamma$ to be a smooth rank-$m$ subbundle of the ambient tangent bundle over $\gamma$.

Let $\mathscr{D}$ be a distribution of rank $m$ along $\gamma$ such that $\dot{\gamma}(t) \in \mathscr{D}_{t}$ for all $t \in I$. The standard Euclidean metric $\overline{g}$ on $\mathbb{R}^{m+n}$ allows us to decompose $T\mathbb{R}^{m+n}\lvert_{\gamma}$ into the orthogonal direct sum of $\mathscr{D}$ and its normal bundle $\mathscr{D}^{\perp}$; indeed, letting $\mathscr{D}^{\perp}_{t}$ denote the orthogonal complement of $\mathscr{D}_{t} \subset T_{\gamma(t)}\mathbb{R}^{m+n}$ with respect to $\overline{g}$, define $\mathscr{D}^{\perp} = \bigcup_{t}\mathscr{D}^{\perp}_{t}$, and so
\begin{equation} \label{DAC-EQ1}
T\mathbb{R}^{m+n}\lvert_{\gamma} \:= \mathscr{D} \oplus \mathscr{D}^{\perp}.
\end{equation}

In this setting, if $v$ is an element of $T\mathbb{R}^{m+n}\lvert_{\gamma}$, then the \textit{tangential projection} is the map $\pi^{\top}\colon T\mathbb{R}^{m+n}\lvert_{\gamma} \: \to \mathscr{D}$ defined by
\begin{equation*}
v \mapsto \pi^{\top}(v),
\end{equation*}
where $\pi^{\top}(v)$ is the orthogonal projection of $v$ onto $\mathscr{D}$. Likewise, denoting by $\pi^{\perp}(v)$ the orthogonal projection of $v$ onto $\mathscr{D}^{\perp}$, the \textit{normal projection} $\pi^{\perp}$ is the map $T\mathbb{R}^{m+n}\lvert_{\gamma}\: \to \mathscr{D}^{\perp}$ defined by
\begin{equation*}
v \mapsto \pi^{\perp}(v).
\end{equation*}

Let now $(E_{1}, \dotsc, E_{m})$ be a smooth $\gamma$-adapted orthonormal frame for $\mathscr{D}$: this is just an $m$-tuple of smooth vector fields along $\gamma$ such that $E_{1} = \dot{\gamma}$ and such that $(E_{i}(t))_{i=1}^{m}$ is an orthonormal basis of $\mathscr{D}_{t}$ for all $t$. Similarly, let $(N_{1}, \dotsc, N_{n})$ be an orthonormal frame for $\mathscr{D}^{\perp}$. It follows that the $(m+n)$-tuple $(E_{1},\dotsc,E_{m},N_{1},\dotsc,N_{n})$ is an orthonormal frame along $\gamma$ that respects the direct sum decomposition \eqref{DAC-EQ1}.

Hence, denoting by $\overline{D}_{t}$ the Euclidean covariant derivative along $\gamma$, i.e., the covariant derivative along $\gamma$ determined by the Levi-Civita connection of $(\mathbb{R}^{m+n},\overline{g})$, we may write:
\begin{equation} \label{DAC-EQ2}
\overline{D}_{t}E_{i} = \pi^{\top}(\overline{D}_{t}E_{i}) + \tau_{i}^{1}N_{1} + \dotsb + \tau_{i}^{n}N_{n}.
\end{equation}
Here $\tau_{i}^{1} ,\dotsc,\tau_{i}^{n}$ are smooth functions $I \to \mathbb{R}$. In particular, indicating $\overline{g}$ by a dot, $\tau_{i}^{k} = \overline{D}_{t}E_{i} \cdot N_{k}$.

\section{The developability condition} \label{DevelopabilityCondition}

In this section we aim to generalize a well-known result about ruled surfaces in $\mathbb{R}^{2+n}$.

\begin{lemma}[\cite{yano1944}] \label{DS-LM1}
Let $I$, $J$ be intervals. Further, let $\gamma$ and $X$ be curves $I \to \mathbb{R}^{2+n}$ such that the map $\sigma \colon I \times J \to \mathbb{R}^{2+n}$ given by
\begin{equation*}
\sigma(t,u) = \gamma(t)+ u X(t)
\end{equation*}
is a smooth embedding. Then the tangent space of $\sigma$ is constant along each ruling precisely when $\gamma$ and $X$ satisfy $\dot{\gamma} \wedge \dot{X} \wedge X = 0$.
\end{lemma}

To begin with, we shall extend the classical notion of ruled surface to arbitrary dimension and codimension.
\begin{definition} \label{DS-DEF2}
An $m$-dimensional embedded submanifold $M^{m}$ of $\mathbb{R}^{m+n}$ is a \textit{ruled} submanifold if
\begin{enumerate}
\item \label{cond1} $M$ is free of planar points, that is, there exists no point of $M$ where the second fundamental form vanishes;
\item there exists a \textit{ruled structure on} $M$, that is, a foliation of $M$ by open subsets of $(m-1)$-dimensional affine subspaces of $\mathbb{R}^{m+n}$, called \textit{rulings}. 
\end{enumerate}
\end{definition}

Following \cite{ushakov1999}, we now define rank-one submanifolds and give three alternative characterizations of them.

\begin{definition} \label{DS-DEF1.3}
The \textit{relative nullity index} of $M^{m}$ at a point $p$ is the dimension of the \textit{nullity space} $\Delta$ of $M$ at $p$, which is the kernel of the second fundamental form $\alpha$ of $M$ at $p$:
\begin{equation*}
\Delta= \{ x \in T_{p}M \mid \alpha(x,\cdot)=0 \}.
\end{equation*}
We say that $M$ is a \textit{rank-one} (or \textit{single-rank}) submanifold if, for all $p \in M$, the relative nullity index is equal to $m-1$.
\end{definition} 

\begin{theorem}[{\cite[Lemma~3.1]{hartman1965}}] \label{DS-TH1.4}
Let $M$ be an $m$-dimensional embedded submanifold of $\mathbb{R}^{m+n}$. Let $\iota \colon M \hookrightarrow \mathbb{R}^{m+n}$ denote inclusion, and let $\overline{g}$ be the standard Euclidean metric on $\mathbb{R}^{m+n}$. The following statements are equivalent:
\begin{enumerate}[font=\upshape]
\item $M$ is rank-one.
\item The Gauss map $M \to G(m, \mathbb{R}^{m+n})$ has rank one everywhere.
\item $M$ is ruled and for every pair of points $(p,q)$ on the same ruling we have $T_{p}M = T_{q}M$, i.e., all tangent spaces of $M$ along any fixed ruling can be canonically identified with the same linear subspace of $\mathbb{R}^{m+n}$.
\item $M$ is ruled and the induced metric on $M$ is flat, that is, the Riemannian manifold $(M, \iota^{\ast}\overline{g})$ is locally isometric to $(\mathbb{R}^{m}, \overline{g})$.
\end{enumerate}
\end{theorem}

Note that if $n=1$, then the theorem still holds when the requirement ``$M$ is ruled'' in the fourth statement is replaced by ``$M$ is free of planar points''. In other words, any flat hypersurface without planar points is automatically ruled.

Given a curve $\gamma$ in $\mathbb{R}^{m+n}$, the following result is the key for constructing ruled submanifolds containing $\gamma$. 

\begin{lemma} \label{DS-LM5}
Let $\gamma \colon I \to \mathbb{R}^{m+n}$ be a smooth embedding, with $I$ closed. Let $(X_{1}, \dotsc, X_{m-1})$ be a smooth $(m-1)$-tuple of vector fields along $\gamma$ such that $\dot{\gamma}(t) \wedge X_{1}(t) \wedge \dotsb \wedge X_{m-1}(t) \neq 0$ for all $t \in I$. Then there exists an open box $V$ in $\mathbb{R}^{m-1}$ containing the origin such that the restriction to $I \times V$ of the map $\sigma  \colon I \times \mathbb{R}^{m-1} \to \mathbb{R}^{m+n}$ defined by
\begin{equation*}
\sigma(t, u) = \gamma(t) + u^{j} X_{j}(t)
 \end{equation*}
is a smooth embedding.
\end{lemma}
\begin{proof}[Proof \textup{(\cite[proof of Lemma 4.6]{markina2019})}]
To show that $\sigma$ restricts to an embedding, we first prove the existence of an open box $V_{1}$ such that $\sigma \rvert_{I\times V_{1}}$ is a smooth immersion. The statement will then follow by the local slice criterion for embedded submanifolds \cite[Theorem~5.8]{lee2013}.

Clearly, the map $\sigma$ is immersive at $(t,u)$ if and only if the length $\ell \colon I \times \mathbb{R}^{m-1} \to \mathbb{R}$ of the wedge product of the partial derivatives of $\sigma$ is nonzero at $(t,u)$. Thus, let $W_{t}$ be the subset of $\{t\} \times \mathbb{R}^{m-1}$ where $\sigma$ is immersive. It is an open subset in $\mathbb{R}^{m-1}$, because it is the inverse image of an open set under a continuous map, $W_{t} = \ell(t,\cdot)^{-1}(\mathbb{R} \setminus \{0\})$; it contains $0$ by assumption. Hence there exists an $\epsilon_{t} >0$ such that the open ball $B(\epsilon_{t},0) \subset \mathbb{R}^{m-1}$ is completely contained in $W_{t}$. Letting $\epsilon_{1} = \inf_{t \, \in \, I}(\epsilon_{t})$, we can conclude that the restriction of $\sigma$ to the box $I\times (-\epsilon_{1}/2,\epsilon_{1}/2)^{m-1}$ is a smooth immersion.

Now, being $\sigma$ a smooth immersion on $I \times V_{1}$, it follows that every point of $I \times V_{1}$ has a neighborhood on which $\sigma$ is a smooth embedding. Besides, passing to a smaller $\epsilon_{1}$ if needed, we may assume that the image of any such neighborhood is the intersection of an open set in $\mathbb{R}^{m+n}$ with $\sigma(I \times V_{1})$. Hence $\sigma(I \times V_{1})$ satisfies the local slice criterion for embedded submanifolds, as desired.
\end{proof}

\begin{remark}
If $I$ is not closed, then we can still choose an open set $V$ on which $\sigma$ restricts to a smooth embedding. However, we can no longer demand $V$ to be a box, for the infimum $\epsilon_{1}$ might be zero.
\end{remark}

We have thus verified that $\sigma\rvert_{I\times V}$ is an $m$-dimensional embedded submanifold of $\mathbb{R}^{m+n}$ and $\{ \sigma(t,V) \}_{t \, \in \, I}$ a ruled structure on it. Letting
\begin{equation} \label{DS-EQ1}
Z = \frac{\partial \sigma}{\partial t} \wedge \frac{\partial \sigma}{\partial u^{1}} \wedge \dotsb \wedge \frac{\partial \sigma}{\partial u^{m-1}},
\end{equation}
we may express the constancy of the tangent space of $\sigma\rvert_{I\times V}$ along the coordinate vector field $\frac{\partial \sigma}{\partial u^{j}}(t,\cdot)$ as follows: for each value of $u^{j}\neq0$ there exists a (nonzero) real number $\lambda$ such that
\begin{equation*}
\lambda Z(t,0)= Z(t,0,\dotsc,0,u^{j},0,\dotsc,0).
\end{equation*}

The next lemma translates this condition into an equation involving the vector fields $X_{1}, \dotsc, X_{m-1}$ along $\gamma$. As an easy corollary we obtain the desired generalization of Lemma \ref{DS-LM1}.
\begin{lemma} \label{DS-LM6}
The tangent space of $\sigma\rvert_{I\times V}$ is constant along $\frac{\partial \sigma}{\partial u^{j}}$ if and only if
\begin{equation} \label{DS-EQ2}
\overline{D}_{t}X_{j} \wedge \dot{\gamma} \wedge X_{1} \wedge \dotsb \wedge X_{m-1} = 0 .
\end{equation}
\end{lemma}
\begin{proof}
Computing the partial derivatives of $\sigma$ and substituting them into the expression \eqref{DS-EQ1} of $Z$, we obtain
\begin{equation*}
Z(t,u)= \bigl\{ \dot{\gamma}(t) + u^{j} \overline{D}_{t}X_{j}(t)\bigr\} \wedge X_{1}(t) \wedge \dotsb \wedge X_{m-1}(t).
\end{equation*}
Hence, we need to show that
\begin{equation} \label{DS-EQ3}
\overline{D}_{t}X_{j}(t) \wedge \dot{\gamma}(t) \wedge X_{1}(t) \wedge \dotsb \wedge X_{m-1}(t) = 0 
\end{equation}
if and only if for each $u^{j}\neq 0$ there exists $\lambda$ such that
\begin{equation*}
(\lambda-1) \dot{\gamma}(t) \wedge X_{1}(t) \wedge \dotsb \wedge X_{m-1}(t) = u^{j} \overline{D}_{t}X_{j}(t) \wedge X_{1}(t) \wedge \dotsb \wedge X_{m-1}(t).
\end{equation*}

First, assume that for some $u^{j}\neq 0$ such a $\lambda$ exists. If $\lambda=1$, then the $m$-vector on the right hand side is necessarily zero. Else, if $\lambda \neq 1$, then the tuples $(\dot{\gamma}(t),X_{1}(t),\dotsc,X_{m-1}(t))$ and $(\overline{D}_{t}X_{j}(t), X_{1}(t),\dotsc,X_{m-1}(t))$ have the same span. Either way, it is clear that \eqref{DS-EQ3} holds.

Conversely, if the vectors $\overline{D}_{t}X_{j}(t), \dot{\gamma}(t),X_{1}(t),\dotsc,X_{m-1}$ are linearly dependent, then there exist real numbers $a_{1},\dotsc,a_{m}$ such that
\begin{equation*}
\overline{D}_{t}X_{j}(t) = a_{1}X_{1}(t)+\dotsb+a_{m-1}X_{m-1}(t)+a_{m}\dot{\gamma}(t).
\end{equation*}
It follows that
\begin{equation*}
u^{j} \overline{D}_{t}X_{j}(t) \wedge X_{1}(t) \wedge \dotsb \wedge X_{m-1}(t) = u^{j}a_{m}\dot{\gamma}(t)\wedge X_{1}(t) \wedge \dotsb \wedge X_{m-1}(t),
\end{equation*}
and so for any $u^{j}\neq 0$ the desired $\lambda$ satisfies $\lambda -1=u^{j}a_{m}$.
\end{proof}

\begin{corollary} \label{DS-COR7}
$\sigma\rvert_{I\times V}$ is rank-one if and only if it is ruled (i.e., without planar points) and the following $m-1$ equations are fulfilled:
\begin{align*}
\overline{D}_{t}X_{1} \wedge \dot{\gamma} \wedge X_{1} \wedge \dotsb \wedge X_{m-1} &= 0 ,\\
&\mathrel{\makebox[\widthof{=}]{\vdots}} \\
\overline{D}_{t}X_{m-1} \wedge \dot{\gamma} \wedge X_{1} \wedge \dotsb \wedge X_{m-1} &= 0.
\end{align*}
\end{corollary}

As a final result of this section we prove the following proposition, which will be useful in the proof of Theorem \ref{IMR-TH3}.

\begin{proposition}\label{developabilityPROP}
If $X_{j}$ is tangent to $\mathscr{D}$ along $\gamma$---i.e., $X_{j}(t) \in \mathscr{D}_{t}$ for every $t \in I$---then \eqref{DS-EQ2} is equivalent to 
\begin{equation} \label{DS-EQ4}
X_{j}^{i} \tau_{i}^{1} = \dotsb = X_{j}^{i} \tau_{i}^{n} = 0,
\end{equation}
where $X_{j}^{i}$ denotes the $i$th coordinate function of $X_{j}$ with respect to $(E_{1},\dotsc, E_{m})$, and where $\tau_{i}^{1}, \dotsc \tau_{i}^{n}$ are defined by \eqref{DAC-EQ2}.
\end{proposition}
\begin{proof}
Clearly---assuming that $X_{j}$ is tangent to $\mathscr{D}$ along $\gamma$---equation \eqref{DS-EQ2} holds if and only if $\overline{D}_{t}X_{j} \cdot N_{k}=0$ for all $k \in \{1,\dotsc,n\}$.
Differentiating $X_{j} = X_{j}^{i}E_{i}$ and substituting in $\overline{D}_{t}X_{j} \cdot N_{k}=0$, we obtain
\begin{equation*}
X_{j}^{i} \overline{D}_{t}E_{i} \cdot N_{k}= 0 ,
\end{equation*}
for the term $\dot{X}_{j}^{i}E_{i} \cdot N_{k}$ vanishes. Using \eqref{DAC-EQ2}, this is equivalent to 
\begin{equation*}
X_{j}^{i}\left(\tau_{i}^{1}N_{1} + \dotsb + \tau_{i}^{n}N_{n} \right) \cdot N_{k}=0 ,
\end{equation*}
again because $X_{j}^{i} \pi^{\top}(\overline{D}_{t}E_{i}) \cdot N_{k} = 0$.
\end{proof}

\begin{remark}
Theorem \ref{IMR-TH3}\ref{item2} can easily be obtained as a corollary of Lemma \ref{DS-LM6}: As already mentioned in the proof of Proposition \ref{developabilityPROP}, when $X_{j}$ is tangent to $\mathscr{D}$ along $\gamma$, equation \eqref{DS-EQ2} is equivalent to $\overline{D}_{t}X_{j} \cdot N_{k} = 0$ for all $k \in \{1,\dotsc, n\}$. Moreover, since $X_{j}$ and $N_{k}$ are orthogonal, $\overline{D}_{t}X_{j} \cdot N_{k} = 0$ if and only if $\overline{D}_{t}N_{k} \cdot X_{j} = 0$, which in turn is equivalent to $\rho(N_{k})\cdot X_{j} =0$.
\end{remark}

\section{Proof of Theorem \ref{IMR-TH3}} \label{ProofMainTh}

We are now ready to finalize the proof of Theorem \ref{IMR-TH3}. Before treating the general case, let us for now assume $n=1$. To simplify notation, in this section we often write $\tau_{i}$ as a shorthand for $\tau_{i}(t)$.

Let $(x_{1}, \dotsc, x_{m-1})$ be a linearly independent $(m-1)$-tuple of vectors in $\mathscr{D}_{t} = \Span(e_{i})_{i=1}^{m}$, where $e_{i} = E_{i}(t)$. Denoting by $x_{j}^{i}$ the $i$th coordinate of $x_{j}$ with respect to the basis $(e_{i})_{i=1}^{m}$, we may identify the tuple $(x_{1}, \dotsc, x_{m-1})$ with the matrix
\[
X=
\begin{pmatrix}
x_{1}^{1} & \dots & x_{1}^{m} \\
\vdots & \ddots & \vdots \\
x_{m-1}^{1} & \dots & x_{m-1}^{m}
\end{pmatrix} \in \mathcal{A}^{(m-1)\times m}.
\]

According to Lemma \ref{DS-LM5} and Proposition \ref{developabilityPROP}, the problem is to find $[X] =[x_{1}, \dotsc, x_{m-1}] \in \mathcal{A}^{(m-1)\times m}/{\sim}$ such that, for every $j \in \{1,\dotsc,m-1\}$, both the conditions 
\begin{align}
&x_{j}^{i}\tau_{i} = 0, \label{PMR-EQ1} \\
&e_{1} \wedge x_{1} \wedge \dotsb \wedge x_{m-1} \neq 0 \label{PMR-EQ2}
\end{align}
are satisfied. (Since we are assuming $n=1$, we write $\tau_{i}$ for $\tau_{i}^{1}$.)

First, we shall examine \eqref{PMR-EQ2}. It is easy to see that \eqref{PMR-EQ2} corresponds to the requirement that the $(m-1)\times(m-1)$ submatrix $X_{2\dotsm m-1}$ of $X$ obtained by removing the first column of $X$ has full rank; indeed, denoting by $S_{2}^{m}$ the group of permutations $\sigma$ of $(2,\dotsc,m)$,
\begin{align*}
	e_{1} \wedge x_{1} \wedge \dotsb \wedge x_{m-1} &= e_{1}\wedge x_{1}^{i}e_{i} \wedge \dotsb \wedge x_{m-1}^{i}e_{i} \\
	&= e_{1} \wedge x_{1}^{2} e_{2} + \dotsb +x_{1}^{m} e_{m} \wedge \dotsb \wedge x_{m-1}^{2} e_{2} + \dotsb +x_{m-1}^{m} e_{m}\\
	&= \sum_{\sigma \in S_{2}^{m}} \Sign(\sigma) x_{1}^{\sigma(2)} \dotsm x_{m-1}^{\sigma(m)}\,  e_{1}\wedge \dotsb \wedge e_{m}\\
	&= \det X_{2\dotsm m-1} \, e_{1}\wedge \dotsb \wedge e_{m} .
\end{align*}
Hence, we just need to look for $[X] \in U_{2\dotsm m-1}=V_{2\dotsm m-1}/{\sim}$ such that, for every $j$, equation \eqref{PMR-EQ1} holds.

Define a map $\psi_{2\dotsm m-1} \colon \mathbb{R}^{m-1} \to V_{2\dotsm m-1}$ by
\[
z=(z^{1},\dotsc, z^{m-1}) \mapsto 
\begin{pmatrix}
z^{1} & 1 & 0 & \dots & 0\\
z^{2} & 0 & 1 & \dots & 0\\
\vdots & \vdots & \vdots & \ddots & \vdots\\
z^{m-1} & 0 & 0 & \dots & 1
\end{pmatrix} .
\]
Since, by definition, $\phi_{2\dotsm m-1}$ maps $[A] \in U_{2\dotsm m-1}$ to the first column $A_{2\dotsm m-1}^{-1}A_{1}$ of
\begin{equation*}
	A_{2\dotsm m-1}^{-1}A=
	\begin{pmatrix}
		* & 1 & 0 & \dots & 0\\
		* & 0 & 1 & \dots & 0\\
		\vdots & \vdots & \vdots & \ddots & \vdots\\
		* & 0 & 0 & \dots & 1
	\end{pmatrix},
\end{equation*}
it is clear that $\pi \circ\psi_{2\dotsm m-1}=\phi_{2\dotsm m-1}^{-1}$. Being $\phi_{2\dotsm m-1}^{-1}$ a parametrization of $U_{2\dotsm m-1}$, our solution set 
\begin{equation*}
	\left\{ [X] \in U_{2\dotsm m-1} \mid x_{j}^{i} \tau_{i} =0 \text{ for all $j$} \right\}
\end{equation*}
coincides with
\begin{equation*}
	\left\{ [\psi_{2\dotsm m-1}(z)]\mid z \in \mathbb{R}^{m-1} \text{ and } \psi_{2\dotsm m-1}(z)_{j}^{i}\tau_{i}=0 \text{ for all $j$} \right\}.
\end{equation*}
In other words, the original problem in $[X]$ reduces to the uncoupled system of equations $\{\psi_{2\dotsm m-1}(z)_{j}^{i}\tau_{i}=0\}_{j}= \{z^{j}\tau_{1}+\tau_{j+1}=0\}_{j}$ on $\mathbb{R}^{m-1}$.

Assume $\tau_{1} \neq0$. Then $z^{j} = -\tau_{j+1}/\tau_{1}$. Since $[X]=[-\tau_{1}X]$, it follows that the matrix
\begin{equation*}
	\begin{pmatrix}
		\tau_{2} & -\tau_{1} & 0 & \dots & 0\\
		\tau_{3} & 0 & -\tau_{1} & \dots & 0\\
		\vdots & \vdots & \vdots & \ddots & \vdots\\
		\tau_{m} & 0 & 0 & \dots & -\tau_{1}
	\end{pmatrix},
\end{equation*}
or equivalently the tuple $(\tau_{j+1} e_{1} - \tau_{1} e_{j+1})_{j}$, represents the unique solution of our problem.

Let us now consider the case where $n$ is arbitrary. Equation \eqref{PMR-EQ1} turns into the system
\begin{equation*} \label{PMR-EQ3}
x_{j}^{i}\tau_{i}^{1}= \dotsb = x_{j}^{i}\tau_{i}^{n}=0.
\end{equation*}

Assume $\pi^{\perp}(\overline{D}_{t}E_{1})(t) \neq 0$. It follows from \eqref{DAC-EQ2} that there exists $s \in \{1, \dotsc, n\}$ such that $\tau_{1}^{s} \neq 0$. The $s$th system $\{x_{j}^{i}\tau_{i}^{s}=0\}_{j}$ admits therefore the unique solution $(\tau_{j+1}^{s} e_{1} - \tau_{1}^{s} e_{j+1})_{j}$. Clearly, the solution satisfies the $k$th system $\{x_{j}^{i}\tau_{i}^{k}=0 \}_{j}$ if and only if $\{ \tau_{j+1}^{s}\tau_{1}^{k}-\tau_{1}^{s}\tau_{j+1}^{k}=0 \}_{j}$; explicitly,
\begin{equation}\label{kSystemCondition}
	\tau_{2}^{k}= \frac{\tau_{1}^{k}}{\tau_{1}^{s}} \tau_{2}^{s}, \; \dots, \; 
	\tau_{m}^{k} =\frac{\tau_{1}^{k}}{\tau_{1}^{s}} \tau_{m}^{s}.
\end{equation}
Being \eqref{kSystemCondition} equivalent to 
\begin{align*}
	\pi^{\top} (\overline{D}_{t} N_{k})(t) &= -\tau_{1}^{k} e_{1} - \dotsb - \tau_{m}^{k}e_{m}\\
	&=\frac{\tau_{1}^{k}}{\tau_{1}^{s}} \left(-\tau_{1}^{s} e_{1} - \dotsb - \tau_{m}^{s}e_{m} \right)\\
	&=\frac{\tau_{1}^{k}}{\tau_{1}^{s}}\, \pi^{\top} (\overline{D}_{t} N_{s})(t),
\end{align*}
we conclude that $(\tau_{j+1}^{s} e_{1} - \tau_{1}^{s} e_{j+1})_{j}$ solves the remaining $n-1$ systems precisely when the rank of $\rho_{t}$ is one.

\section{Codimension reduction} \label{CodimensionReduction}

In this section we present a sufficient condition for the solution of our geometric Cauchy problem to admit a reduction of its codimension.

Let $M^{m}$ be a submanifold of a Riemannian manifold $\widetilde{M}^{m+n}$. Recall that $M$ is said to be a \textit{full submanifold} if it is not contained in any totally geodesic submanifold $S$ of $\widetilde{M}$ with $\dim S < \dim \widetilde{M}$. If $M$ is not full, then one says that there is a \textit{reduction of the codimension of} $M$ \cite[p.~16]{berndt2016}. 

A key result about codimension reduction was given by Erbacher in 1971.

\begin{theorem}[\cite{erbacher1971}] \label{CR-TH1}
Assume that $\widetilde{M}^{m+n}$ has constant sectional curvature. If the first normal space of $M$ is invariant under parallel translation with respect to the normal connection and has constant dimension $l$, then $M$ is not full. In particular, $M$ is contained in an $(m+l)$-dimensional totally geodesic submanifold of $\widetilde{M}^{m+n}$.
\end{theorem}

Let $N_{p}M$ be the normal space of $M$ at $p$. Recall that the \textit{first normal space} of $M$ at $p$ is the linear subspace $N_{p}^{1}M$ of $N_{p}M$ spanned by the image of the second fundamental form at $p$. In other words, $N_{p}^{1}M$ is the orthogonal complement in $N_{p}M$ of the kernel of the linear map $N_{p}M \to \mathrm{End}(T_{p}M)$, $\nu \mapsto A_{\nu}$, where $A_{\nu}$ denotes the shape operator of $M$ in direction $\nu$. If the dimension of $N_{p}^{1}M$ is constant on $M$, then $N^{1}M$ is a smooth subbundle of the normal bundle of $M$.

Since $\Ima \alpha = \{\alpha(x,y) \mid x,y \in T_{p}M \ominus \Delta\}$, it is clear that the dimension of the first normal space at any point of a rank-one submanifold is one. Hence, using Theorem \ref{CR-TH1} (see also Remark \ref{IMR-RM4}), we obtain the following proposition.

\begin{proposition}
Assume that the function $\pi^{\perp}(\overline{D}_{t}\dot{\gamma})$ is never zero and there exists a smooth orthonormal frame $(N^{\ast}_{1},\dotsc,N^{\ast}_{n})$ for $\mathscr{D}^{\perp}$ such that $\pi^{\top}(\overline{D}_{t}N^{\ast}_{k})=0$ for all but one value $s$ of $k\in \{1,\dotsc,n\}$. If $\pi^{\perp}(\overline{D}_{t}N^{\ast}_{s})=0$, then the solution of the associated geometric Cauchy problem for rank-one submanifolds of $\mathbb{R}^{m+n}$ is not full. In particular, it is a hypersurface in an $(m+1)$-dimensional affine subspace of $\mathbb{R}^{m+n}$.
\end{proposition}

\section{Application to approximations} \label{Application}

Here we consider the problem of approximating, in a neighborhood of a curve, an arbitrary submanifold by a single-rank one.

Let $M^{m}$ be a submanifold of $\mathbb{R}^{m+n}$, and let $\gamma$ be a smooth regular curve in $M$. We say that a submanifold is a (\textit{first-order}) \textit{approximation of $M$ along $\gamma$} if it contains $\gamma$ and has the same tangent bundle as $M$ along $\gamma$.

The result below follows easily from our main theorem.

\begin{theorem} \label{APP-TH1}
Let $\alpha$ be the second fundamental form of $M$. Suppose that the curve $\gamma$ is never parallel to an asymptotic direction of $M$, i.e., that $\alpha(\dot{\gamma}, \dot{\gamma})$ never vanishes. Suppose further that the linear map $\alpha_{t} = \alpha(\dot{\gamma}(t), \cdot)$ has rank one for all $t \in I$. Then there exists a rank-one approximation of $M$ along $\gamma$. Such approximation is locally unique, and may be constructed as presented in Theorem~\textup{\ref{IMR-TH3}\ref{item3}}.
\end{theorem}

\begin{proof}
With the notations of Theorem \ref{IMR-TH3}, set $\mathscr{D}_{t}=T_{\gamma(t)}M$ and assume that $\alpha_{t} \colon x \mapsto \pi^{\perp}\overline{D}_{t}x$ has rank one. We shall show that $\Rank \rho_{t} =1$.

Note that if $x \in \mathscr{D}_{t}$ and $\nu \in \mathscr{D}_{t}^{\perp}$, then
\begin{equation*}
\rho_{t}(\nu) \cdot x = \overline{D}_{t}\nu \cdot x = - \nu\cdot \overline{D}_{t}x = - \nu \cdot \alpha_{t}(x). 
\end{equation*}
From this we conclude that $\rho_{t}$ and $\alpha_{t}$ are negative adjoint with respect to the dot product, and so they have the same rank.
\end{proof}

\begin{remark}
In the case where $n=1$, the condition that the curve $\gamma$ is never parallel to an asymptotic direction of $M$ becomes sufficient for the existence of a developable approximation of $M$ along $\gamma$. It follows that, on a positively curved hypersurface, such approximation always exists, regardless of the choice of curve. Similarly, in higher codimension, the existence of a rank-one approximation becomes trivial on a positively curved submanifold whose second fundamental form has \emph{rank one in every direction}.
\end{remark}

\section{A local description of rank-one submanifolds} \label{Description}

So far, we have studied rank-one submanifolds, and yet have not produced any. It is thus natural to ask whether they are plentiful or rare. In this section we shall address this question by proving the following theorem.
\begin{theorem} \label{LD-TH1}
Let $b_{1}, \dotsc, b_{m+n}$ be the standard basis vectors of $\mathbb{R}^{m+n}$. For any $(m+n-1)$-tuple of smooth functions $(f_{1}, \dotsc, f_{m-1}, g_{1}, \dotsc, g_{n})$ on $I$, where $g_{1}, \dotsc, g_{n}$ are nowhere vanishing, the solution of the linear ordinary differential equation problem
\begin{equation} \label{D-EQ1}
\begin{cases}
\overline{D}_{t} E_{1} =  -f_{1}E_{2} - \dotsb -f_{m-1} E_{m} -g_{1} N_{1} - \dotsb -g_{n} N_{n}, \\
\overline{D}_{t} E_{2} = f_{1} E_{1}, \\
\mathrel{\makebox[\widthof{=}]{\vdots}} \\
\overline{D}_{t} E_{m} = f_{m-1} E_{1}, \\
\overline{D}_{t}N_{k} = g_{k} E_{1}, \quad k = 1,\dotsc,n,\\ 
\left(E_{i}(0)\right)_{i=1}^{m} = (b_{1},\dotsc, b_{m}),\\
\left(N_{k}(0)\right)_{k=1}^{n} = (b_{m+1},\dotsc, b_{m+n})
\end{cases}
\end{equation}
defines a full rank-one submanifold of $\mathbb{R}^{m+n}$. Any such submanifold can be locally represented in this way.
\end{theorem}

\begin{proof}
Let $M$ be a full rank-one submanifold, and let $\gamma \colon [0,\alpha] = I \to M$ be a smooth unit-speed curve orthogonal to the rulings. Since we are free to move $M$ rigidly as we please, there is no loss of generality in assuming $\dot{\gamma}(0) = b_{1}$ and $T_{\gamma(0)}M = \Span(b_{1},\dotsc,b_{m})$. 

Letting $E_{1} = \dot{\gamma}$, we define an orthonormal frame $(E_{i},N_{k})_{i,k=1}^{m,n}$ for the ambient tangent bundle over $\gamma \colon I \to \mathbb{R}^{m+n}$ as follows: first, we extend $(b_{2},\dotsc,b_{m})$ to an orthonormal parallel frame $(E_{2},\dotsc,E_{m})$ for the normal bundle of $\gamma \colon I \to M$; then, we extend $(b_{m+k})_{k=1}^{n}$ to an orthonormal parallel frame $(N_{k})_{k=1}^{n}$ for the normal bundle of $M$.

Now, since $\gamma$ is orthogonal to the rulings, developability of $M$ implies that $\overline{D}_{t}N_{k} \cdot E_{2} = \dotsb = \overline{D}_{t}N_{k} \cdot E_{m}=0$ for every $k$. Moreover, as $M$ is assumed to be full, $\overline{D}_{t}N_{k} \cdot E_{1}(t) \neq 0$ for every $k$ and $t$. In conclusion, the chosen frame must satisfy \eqref{D-EQ1} for some $(m+n-1)$-tuple of smooth functions $(f_{1}, \dotsc, f_{m-1}, g_{1}, \dotsc, g_{n})$.

Conversely, for any choice of $(f_{1}, \dotsc, f_{m-1}, g_{1}, \dotsc, g_{n})$, problem \eqref{D-EQ1} has unique global solution, thus defining a full rank-one submanifold up to a rigid motion of $\mathbb{R}^{m+n}$.
\end{proof}

\section{Existence of distributions along \texorpdfstring{$\gamma$}{gamma} with \texorpdfstring{$\Rank \rho =1$}{rank rho = 1}} \label{DevDistr}

Although in the last section we have presented a method for constructing examples of rank-one submanifolds, such approach gives virtually no control over the curve $\gamma$, which is obtained by integrating the function $E_{1}$.

In contrast, here we aim to give some insight into the following problem.

\begin{problem} \label{EDD-PR1}
Let $\gamma \colon [0,\alpha] = I \to \mathbb{R}^{m+n}$ be a smooth unit-speed curve whose curvature never vanishes. Let $\mathscr{D}_{0}$ be an $m$-dimensional subspace of $T_{\gamma(0)}\mathbb{R}^{m+n}$ such that $\dot{\gamma}(0)\in \mathscr{D}_{0}$ and $\ddot{\gamma}(0) \in \mathscr{D}_{0}^{\perp}$. Find all full rank-one submanifolds $M$ containing $\gamma$ and such that $T_{\gamma(0)}M = \mathscr{D}_{0}$.
\end{problem}

In particular, we will establish, by using a constructive argument, the following existence result.
\begin{theorem}
Assume that $\gamma$ is full. Then Problem \textup{\ref{EDD-PR1}} admits a solution containing $\gamma$ as a geodesic.
\end{theorem}
\begin{proof}
Let $E_{1}= \dot{\gamma}$, and let $N_{1}=\overline{D}_{t}E_{1}/\kappa$, where $\kappa$ is the curvature of $\gamma$. Extend $E_{1}(0)=e_{1}$ and $N_{1}(0)=n_{1}$ to orthonormal bases $(e_{1},\dotsc,e_{m})$, $(n_{1},\dotsc,n_{n})$ of $\mathscr{D}_{0}$ and $\mathscr{D}_{0}^{\perp}$, respectively. Then the system of equations
\begin{equation} \label{EDD-EQ1}
\begin{cases}
\overline{D}_{t}E_{2} = - \left(\overline{D}_{t}N_{1} \cdot E_{2}\right)N_{1},\\ 
\mathrel{\makebox[\widthof{=}]{\vdots}}\\
\overline{D}_{t}E_{m} = - \left(\overline{D}_{t}N_{1} \cdot E_{m}\right)N_{1},\\ 
\overline{D}_{t}N_{2} = - \left(\overline{D}_{t}N_{1} \cdot N_{2}\right)N_{1},\\ 
\mathrel{\makebox[\widthof{=}]{\vdots}}\\
\overline{D}_{t}N_{n} = - \left(\overline{D}_{t}N_{1} \cdot N_{n}\right)N_{1},
\end{cases}
\end{equation}
equipped with the initial condition
\begin{equation*}
	\begin{cases}
(E_{i}(0))_{i=2}^{m} = (e_{i})_{i=2}^{m},\\
(N_{k}(0))_{k=2}^{n} = (n_{k})_{k=2}^{n},
\end{cases}
\end{equation*}
defines a \emph{linear} ordinary differential equation problem in the coordinate functions of $(E_{i})_{i=2}^{m}$ and $(N_{k})_{k=2}^{n}$, whose solution exists uniquely on the entire interval $I$; indeed, the $h$th coordinate function $\dot{E}_{2}^{h}$ of $\overline{D}_{t}E_{2}$ with respect to the standard basis of $\mathbb{R}^{m+n}$ satisfies
\begin{equation*}
	\dot{E}_{2}^{h} = -\left(\dot{N}_{1}^{1}E_{2}^{1}+ \dotsb + \dot{N}_{1}^{m+n}E_{2}^{m+n} \right) N_{1}^{h},
\end{equation*}
so that
\begin{equation*}
	\dot{E}_{2}^{h} =
	-\begin{pmatrix}
		\dot{N}_{1}^{1}N_{1}^{h} & \dots & \dot{N}_{1}^{m+n}N_{1}^{h}
	\end{pmatrix}
	\begin{pmatrix}
		E_{2}^{1}\\
		\vdots\\
		E_{2}^{m+n}
	\end{pmatrix},
\end{equation*}
and analogous equations apply to $E_{3}, \dotsc, E_{m},N_{2},\dotsc,N_{n}$.

It is easy to see, from \eqref{EDD-EQ1}, that any of $E_{2}, \dotsc, E_{m},N_{2},\dotsc,N_{n}$ is orthogonal to both $E_{1}$ and $N_{1}$. Thus, writing $\tau_{i}$ and $\mu_{k}$ as shorthands for $\overline{D}_{t}E_{1}\cdot E_{i}$ and $\overline{D}_{t}N_{1}\cdot N_{k}$, respectively, the frame $(E_{i},N_{k})_{i,k=1}^{m,n}$ satisfies the equation
\begin{equation*}
	\begin{pmatrix}
		\overline{D}_{t}E_{1}\\
		\overline{D}_{t}E_{2}\\
		\vdots\\
		\overline{D}_{t}E_{m}\\
		\overline{D}_{t}N_{1}\\
		\overline{D}_{t}N_{2}\\
		\vdots\\
		\overline{D}_{t}N_{n}
	\end{pmatrix}=
	\begin{pmatrix}
		0 & 0 & \dots & 0 & \kappa & 0 & \dots & 0\\
		0 & 0 & \dots & 0 & \tau_{2} & 0 & \dots & 0\\
		\vdots & \vdots & \ddots & \vdots & \vdots & \vdots & \ddots & \vdots\\
		0 & 0 & \dots & 0 & \tau_{m} & 0 & \dots & 0\\
		-\kappa & -\tau_{2} & \dots & -\tau_{m} & 0 & -\mu_{2} & \dots & -\mu_{n}\\
		0 & 0 & \dots & 0 & \mu_{2} & 0 & \dots & 0\\
		\vdots & \vdots & \ddots & \vdots & \vdots & \vdots & \ddots & \vdots\\
		0 & 0 & \dots & 0 & \mu_{n} & 0 & \dots & 0
	\end{pmatrix}
	\begin{pmatrix}
		E_{1}\\
		E_{2}\\
		\vdots\\
		E_{m}\\
		N_{1}\\
		N_{2}\\
		\vdots\\
		N_{n}
	\end{pmatrix},
\end{equation*}
from which we conclude that $(E_{i},N_{k})_{i,k=1}^{m,n}$ defines a rank-one submanifold $M$. If $\gamma$ is full, then so is $M$.
\end{proof}

\appendix
\section{Alternative proof of Theorem \ref{IMR-TH3}} \label{AltProof}

In this appendix we present a coordinate-free approach for proving the main part of Theorem \ref{IMR-TH3}, up to and including \textup{\ref{item2}}. We also sketch an alternative method, adapted from \cite[section~5]{markina2019}, for obtaining the parametrized solution given in \ref{item3}.

Let $(X_{j})_{j=1}^{m-1}$ be a smooth, linearly independent $(m-1)$-tuple of vector fields---always tangent to $\mathscr{D}$---along $\gamma$. By Corollary \ref{DS-COR7}, we need to find $(X_{j})_{j=1}^{m-1}$ such that, for any section $Y$ of $\mathscr{D}^{\perp}$ and any $j = 1, \dotsc, m-1$,
\begin{equation}
X_{j} \cdot \overline{D}_{t}Y \equiv X_{j} \cdot \pi^{\top}(\overline{D}_{t}Y) \equiv X_{j} \cdot \rho(Y)=0.
\end{equation}

Hence, our problem amounts to finding $ \Sigma=\Span(X_{j})_{j=1}^{m-1}$, satisfying $\dot{\gamma}(t) \notin \Sigma_{t}$ for every $t$ and such that
\begin{equation} \label{APP-EQ11}
\Sigma \subset (\Ima \rho)^{\perp} \cap \mathscr{D}.
\end{equation}
Here by $(\Ima \rho)^{\perp}$ we mean the distribution $(\Ima \rho_{t})^{\perp}_{t \in I}$, where the superscript $^{\perp}$ denotes orthogonal complement in the ambient tangent space.

Assume that $\pi^{\perp}(\overline{D}_{t}\dot{\gamma})(t) \neq 0$ for all $t\in I$. Then there exists a smooth section $N$ of $\mathscr{D}^{\perp}$ such that $\dot{\gamma} \cdot \overline{D}_{t}N = \dot{\gamma} \cdot \pi^{\top}(\overline{D}_{t}N)$ never vanishes. It follows that, for any $t$, $\Rank \rho_{t} \neq 0$ and $\dot{\gamma}(t) \notin (\Ima \rho_{t})^{\perp}$. Since the dimension of the intersection in \eqref{APP-EQ11} equals $m - \Rank \rho_{t}$, it is clear that a solution $\Sigma$ exists if and only if $\Rank \rho =1$, and that such solution is given by equality in \eqref{APP-EQ11}.

As for \ref{item3}, pick an orientation on $\mathscr{D}$. Associated to such a choice (and the natural bundle metric) there is a well-defined Hodge star operator $\star$ on $\mathscr{D}$, which in turn defines a unique $(m-1)$-fold vector cross product on $\mathscr{D}$; see \cite[section~3]{brown1967}. This product acts on tuples of vector fields $X_{1}, \dotsc, X_{m-1}$ \emph{on} $\mathscr{D}$ by
\begin{equation*}
X_{1} \times \dotsb \times X_{m-1} = \star(X_{1} \wedge \dotsb \wedge X_{m-1}).
\end{equation*}

Let $N$ as above. For every $j \in \{1,\dotsc, m-1\}$, let
\begin{equation} \label{APP-EQ12}
X_{j}(N) = \overline{D}_{t}N \times E_{2} \times \dotsb \times\widehat{E_{j+1}} \times\dotsb \times E_{m},
\end{equation}
where the hat indicates that $E_{j+1}$ is omitted, so that the cross product is $(m-1)$-fold for every $j$. Since $\overline{D}_{t}N$ is never in the span of $E_{2}, \dotsc, E_{m}$, it follows that $(X_{1}(N), \dotsc, X_{m-1}(N))$ is linearly independent, i.e., $\Span (X_{j}(N))_{j=1}^{m-1} = \overline{D}_{t}N^{\perp}$. By computing the coordinates of the cross product in \eqref{APP-EQ12} with respect to the frame $(E_{1},\dotsc,E_{m})$, the desired expression of $X_{j}$ is easily obtained. 

\section{Alternative construction for \texorpdfstring{$n=1$}{n = 1}} \label{AltConstr}

For completeness, in this appendix we briefly discuss an alternative method to construct the solution. Such method is only available when the codimension is one.

Let $n=1$, and let $N$ be a continuous section of $\mathscr{D}^{\perp}$ such that $N \cdot N =1$. This section is automatically smooth; it is unique up to a sign. Assuming existence, the solution of the Cauchy problem for developable hypersurfaces is given by the distribution $\overline{D}_{t}N^{\perp} \cap N^{\perp}$. 

Note that if we identify (through parallel translation in $\mathbb{R}^{m+1}$) the vector field $N$ with a curve in the unit sphere $\mathbb{S}^{m}$, then
\begin{equation*}
\overline{D}_{t}N^{\perp} \cap N^{\perp} \rvert_{t} \equiv \dot{N}(t)^{\perp} \cap T_{N(t)}\mathbb{S}^{m}.
\end{equation*}
Hence we may parametrize the solution using any smooth frame for the normal space of $N \colon I \to \mathbb{S}^{m}$.

\section*{Acknowledgments}
The author thanks David Brander, Ruy Tojeiro, and the anonymous referees for helpful suggestions.
An earlier version of this article was included in the author's Ph.D.\ thesis, written at the Technical University of Denmark under the supervision of Steen Markvorsen and the cosupervision of Jakob Bohr. The work was completed at the Centre for Mathematics of the University of Coimbra, to which the author was affiliated at the time of submission.

\bibliographystyle{amsplain}
\bibliography{../biblio/biblio}
\end{document}